\documentclass[11pt]{amsart}
\usepackage{amsmath,amsfonts,latexsym,amssymb,amsthm,mathrsfs,array}

\usepackage{amsfonts}

\newcommand{\CC}{{\mathbb C}}
\newcommand{\ZZ}{{\mathbb Z}}

\def\gg{{\mathfrak{g}}}
\def\hh{{\mathfrak{h}}}
\def\CC{{\mathbb C}}
\def\ZZ{{\mathbb Z}}

\newcommand{\fma}{\overset{\circ}{\mathfrak{g}}}
\newcommand{\m}{{\mathfrak{m}}}
\newcommand{\n}{{\mathfrak{n}}}
\newcommand{\M}{{\mathcal{M}}}
\newcommand{\fmh}{\overset{\circ}{\mathfrak{h}}}

\newcommand{\fpr}{{\overset{\circ}{\Delta}}_{+}}

\newcommand{\C}{{\mathbb C}}
\newtheorem{dfn}{Definition}[section]
\newcommand{\bdfn}{\begin{dfn}\rm}
\newcommand{\edfn}{\end{dfn}}
\newtheorem{thm}[dfn]{Theorem}
\newcommand{\bthm}{\begin{thm}}
\newcommand{\ethm}{\end{thm}}                   
\newtheorem{lmma}[dfn]{Lemma}                   
\newcommand{\blmma}{\begin{lmma}}                   
\newcommand{\elmma}{\end{lmma}}                   
\newtheorem{ppsn}[dfn]{Proposition}
\newcommand{\bppsn}{\begin{ppsn}}
\newcommand{\eppsn}{\end{ppsn}}
\newtheorem{crlre}[dfn]{Corollary}
\newtheorem{rmk}[dfn]{Remark}
\newcommand{\brmk}{\begin{rmk}\rm} 
\newcommand{\ermk}{\end{rmk}}

\def\LL{{\mathfrak L}}

\numberwithin{equation}{section}

\title{Integrable modules for loop  affine-Virasoro algebra}
\author{ S Eswara Rao, Sachin S. Sharma, Sudipta Mukherjee }
\date{}
\begin{document}
\maketitle
\begin{abstract}
In this paper we classify the irreducible integrable modules for the loop affine-Virasoro algebra $((\fma \otimes \CC[t, t^{-1}] \oplus \CC K) \rtimes \mathrm{Vir}) \otimes A$, where
$A$ is a finitely generated commutative associative algebra with unity.

\end{abstract}

\section{Introduction}
Affine Kac-Moody Lie algebra  can be realised as a central extension of a loop algebra with a degree derivation or its fixed point subalgebra, while Virasoro Lie algebra is a  central extension of vector fields on circle. Both Lie algebras are fundamental in Mathematics and Physics.
Virasoro algebra $\mathrm{Vir}$ acts on derived algebra of affine Kac-Moody algebra $[\gg, \gg]$  by derivations and the Lie algebra $[\gg,\gg] \rtimes \mathrm{Vir}$ is also a very important algebra of study as it has applications to conformal field theory, number theory and soliton theory.  The Lie algebra 
$[\gg, \gg] \rtimes \mathrm{Vir}$ with common center has studied in Physics literature \cite{K2, Ku} and also studied in mathematical point of view in \cite{JHL, JH, LQ}. 

On the other hand loop algebras are also studied extensively for simple Lie algebra, affine Kac-Moody Lie algebra and Virasoro algebra.  For a finite dimensional simple Lie algebras $\fma$, a finitely generated commutative associative algebra $A$, the representations of loop algebras $\fma \otimes A$ were studied in \cite{CFK, NSS}. In \cite{R4}, study of irreducible modules for toroidal Lie algebra reduced to that for loop affine algebras. For $A= \CC[t,t^{-1}]$, complete classification of the Harish-Chandra modules for $\mathrm{Vir} \otimes A$ was obtained in \cite{ZRX}. Later this work was generalised for any finitely  generated commutative associative algebra by A. Savage in \cite{SAV}. For $\fma = \mathfrak{sl}_2$, classification of Harish-Chandra modules for $(\fma \otimes \CC[t, t^{-1}] \oplus \CC K) \rtimes \mathrm{Vir}$ with common center was done \cite{JHL}. For any simple Lie algebra $\fma$,  Rao \cite{R1} studied the Lie algebra $((\fma \otimes \CC[t,t^{-1}] \oplus \CC K) \rtimes \mathrm{Vir}) \otimes A$ with common center and classified all its irreducible integrable modules where affine center (hence also Vir center) acts nontrivially. 

In this paper  
we consider the same Lie algebra as in \cite{R1} but with independent centers and classify its irreducible integrable modules with finite dimensional weight spaces. We prove that the representation of $\LL(A) = ((\fma \otimes \CC [t, t^{-1}] \otimes \CC K) \rtimes \mathrm{Vir}) \otimes A$ are controlled by the affine center action.
More precisely, let $V$ be a irreducible integrable $\LL(A)$-module with finite dimensional weight spaces. In the case where affine center acts nontrivially on $V$ then irrespective how Vir center acts, $V$ must be a highest weight or a lowest weight module $V(\psi)$ with respect to a natural triangular decomposition of $\LL(A)$.
Conversely,  for a given integrable irreducible module with finite dimensional weight spaces $V(\psi)$, the condition on $\psi$ is stated. All this work follows from the work of \cite{R1} with minimal modifications. While the case where affine center acts trivially poses stiffer challenges. In this case we prove that 
Vir center is forced to act trivially and prove our main theorem of paper Theorem \ref{thm}.

In \cite{R1}, motivated from Sugawara operators, Rao defined the so called affine central operators on the category $\mathcal{O}$ of $\LL(A)$ modules. Unlike in \cite{R1}, as we have independent affine and Vir centers we generalise those operators in our setting.

The paper is organised as follows. In section \ref{sec2} we start with basic definitions and state the classification result when center acts non-trivially. Section \ref{sec3}
is devoted where affine center $K$ acts trivially. We define category of integrable irreducible modules $ \LL(S, \alpha, \beta, \psi)$ (see Section \ref{sec3} for details).  Main result of the paper is that this category exhausts  all the irreducible integrable representation with finite dimensional weight spaces
when $K$ acts trivially.  This is obtained by considering suitable triangular decomposition of $\LL(A)$ so that highest weight space $M$ with respect to given triangular decomposition is non-zero.  Then we study the highest weight space $M$ and show that weight space of $V$ are uniformly bounded not only as $\LL(A)$-module but also as Vir-module (Proposition \ref{ps1}). In the process we prove that $V$ has only finitely many $\gg(A)$-submodules. This result  plays a key role in the classification part . We use an interesting result due to Futorny \cite{F} to prove that $K \otimes A$ acts trivially on $V$. Next  we prove that there exists a cofinite ideal $R$ such that $(\LL \otimes R) V = 0$ (Proposition \ref{mp1}). Using results in \cite{ZRX} and \cite{SAV}, we prove that the cofinite ideal $R$ can be taken as a maximal ideal and using the result of \cite{R4}, obtain our main result Theorem \ref{thm}. Finally, Section \ref{Sec4} we generalise the affine central operators for $\LL(A)$. 

    We would like to emphasise that the techniques of this paper heavily rely on the integrability of $V$. So the more general problem of  classification of Harish-Chandra (irreducible modules with finite dimensional weight spaces) $\LL(A)$-modules probably requires a different approach.
However, we do hope that this work will bring attention of many researchers towards this more general problem.

\section{Affine Center acting non-trivially} \label{sec2}
\noindent
{\bf{Notations:}} We denote set of integers, rational numbers and complex numbers by $\ZZ, \mathbb{Q}, \CC$ respectively.  Set all nonnegative integers and positive integers are denoted by $\ZZ_{\geq 0}$ and $\ZZ_{>0}$. For any Lie algebra $\mathfrak{m}$ and finitely generated commutative associative algebra $A$ with unit, we always denote $\mathfrak{m} \otimes A$ by $\mathfrak{m}(A)$. Also for any element $a \in A$, and an ideal $I$ of $A$, by $\m(a)$ and $\m(I)$, we always mean $\m \otimes a$ and $\m \otimes I$.\\
\noindent
{\bf{Semidirect Product:}} 
Let $\mathfrak{a}$ and $\mathfrak{b}$ be Lie algebras. Let $\phi: \mathfrak{b} \mapsto \mathrm{Der}(\mathfrak{a})$ be a Lie algebra homomorphism from $\mathfrak{b}$ to a
Lie algebra of derivations of $\mathfrak{a}$. Then semidirect product of $\mathfrak{b}$ with $\mathfrak{a}$ is a Lie algebra denoted by $\mathfrak{a} \rtimes \mathfrak{b} : = \{a +b: a \in \mathfrak{a}, b \in \mathfrak{b}\}$ with the following
bracket operation : $[a_1 +b_1, a_2+ b_2] =  [a_1, a_2] + [b_1, b_2] + \phi(b_1)(a_2) - \phi(b_2)(a_1)$.
\subsection{}
Let $\fma$ be a finite dimensional simple Lie algebra.  Let $\fmh$ be a Cartan subalgebra of $\fma$. 
 Let ${\overset{\circ}{\mathfrak{n}}_{-}} \oplus \fmh \oplus {\overset{\circ}{\mathfrak{n}}_{+}}$ be a triangular decomposition of $\fma$. Let $\alpha_1, \ldots, \alpha_n$ denote
the set of simple roots and $\overset{\circ}{\Delta}, \overset{\circ}{\Delta}_{+}, \overset{\circ}{\Delta}_{-}$ be the set of roots, positive roots and negative roots of $\fma$ respectively. Root space decomposition of $\fma$ is given by
$\fma = \displaystyle{\fmh \oplus \bigoplus_{\alpha \in \overset{\circ}{\Delta}}{\fma_{\alpha}}}$. Let $\langle \cdot, \cdot \rangle$
be symmetric, non-degenerate bilinear form on $\fma$. Let $\gg = \fma \otimes \CC[t,t^{-1}] \oplus \CC K \oplus \C d_0$ be the corresponding affine Lie algebra with
$\hh = \fmh \oplus \CC K \oplus \CC d_0$ be its Cartan subalgebra. Set of roots, positive roots and negative roots of $\gg$ are denoted by  $\Delta, \Delta^{+}, \Delta^{-}$ respectively. Real roots and Imaginary roots of $\gg$ are denoted by $\Delta^{\mathrm{re}}$ and $\Delta^{\mathrm{im}}$ respectively. 
Let $\mathfrak{n}_{-} \oplus \hh \oplus \mathfrak{n}_{+}$ be the standard triangular decomposition of $\gg$. Simple roots and simple co-roots of $\gg$ are denoted by $\alpha_0, \alpha_1, \ldots, \alpha_n$ and $\alpha_0^{\vee}, \alpha_1^{\vee}, \ldots, \alpha_n^{\vee}$ respectively. Root lattices of $\fma$ and
$\gg$ are denoted by $\overset{\circ}{Q}$ and $Q$ respectively, where $\overset{\circ}{Q} = \sum_{i = 1}^{n}{\ZZ\, \alpha_i}$ and $Q = \sum_{i = 0}^{n}{\ZZ\, \alpha_i}$. Let $\overset{\circ}{Q}_{+} = \sum_{i = 1}^{n}{\ZZ_{\geq 0}\, \alpha_i}$ and $Q_{+} = \sum_{i = 0}^{n}{\ZZ_{\geq 0}\, \alpha_i}$.  Let 
$\gg = \displaystyle{ \hh \oplus \bigoplus_{\alpha \in \Delta}{\gg_{\alpha}}}$ be a root space decomposition.
Let $\gg' = [\gg, \gg] =  \fma \otimes \CC[t,t^{-1}] \oplus \CC K$ and $\hh' = \fmh \oplus \CC K$. We recall the bracket operations of $\gg$:
\begin{align*}
[x \otimes t^n, y \otimes t^{m}] & = [x, y] \otimes t^{m+n} + m \delta_{m, -n} \langle x , y \rangle K ;\\
[K, x \otimes t^n ]  = 0 & \,\,  [K, d_0] = 0 ; \\
[d_0, x \otimes t^n]  = n & x \otimes t^n , \,\,\,\, x, y \in \fma, n, m \in \ZZ.
\end{align*}
Let $\mathrm{Vir}$ be the Virasoro algebra with a basis $\{d_n, C : n \in \ZZ\}$ and bracket operations $[d_n, d_m] = (m -n) d_{n +m} + \delta_{m, -n} \frac{n^3 - n}{12} ; [C, d_n] = 0$. Consider the Lie algebra $\LL = \gg' \rtimes \mathrm{Vir}$ with bracket operations
$[d_n, x \otimes t^m] = m x \otimes t^{m+n} ; [C, x \otimes t^n] = 0, [C, K] = 0, x \in \fma, d_n \in \mathrm{Vir}, n, m \in \ZZ $. Let $A$ be a finitely generated commutative associative algebra with unity. This paper is concerned with the Lie algebra
$\LL(A): = (\gg' \rtimes \mathrm{Vir}) \otimes A$ called a loop affine-Virasoro algebra and its irreducible integrable representations. The bracket operations on $\LL(A)$ is given by $[X \otimes a, Y \otimes b] = [X, Y] \otimes ab,  \,\, X, Y \in \LL, a, b \in A$. Let $\bar{\hh} = \hh \oplus \CC C$ be a Cartan subalgebra of $\LL(A)$.  Now we write root space decomposition of $\LL(A)$.
Let $\delta, \Lambda_0, \omega \in \bar{\hh}^*$ defined by $\delta(\fmh \oplus \CC K \oplus \CC C) = 0 ; \delta(d_0) = 1$ and $\Lambda_0 (\fmh \oplus \CC C \oplus \CC d_0) = 0 ; \Lambda_0 (K) = 1$ and $\omega(\fmh \oplus \CC K \oplus \CC d_0) = 0 ; \omega(C) = 1$. We define a  bilinear form
on $\bar{\hh}^*$ by extending $\alpha \in \fmh^*$ to $\bar{\hh}^*$ by defining $\alpha(K) = \alpha(d_0) = \alpha (C) = 0$. Then $\langle \fmh, K \rangle = \langle \delta, \delta \rangle = \langle \Lambda_0, \Lambda_0 \rangle = \langle \omega, \omega \rangle = 0,   \langle \delta, \Lambda_0 \rangle = 1, \langle \delta, \omega \rangle = 1, \langle \Lambda_0, \omega \rangle = 1$ defines a non-degenerate symmetric bilinear form on $\bar{\hh}^*$. Every $\eta \in \bar{\hh}^*$ can be uniquely written as $\eta = \bar{\eta} + \eta(K)\Lambda_0 + \eta(d_0) \delta + \eta(C) \omega$, where
 $\bar{\eta}$ is the image of $\eta$ under the orthogonal projection from $\bar{\hh}^*$ to $\fmh^*$. Define orders on $\fmh^*$ and $\bar{\hh}^*$  by $\bar{\lambda} \leq_0 \bar{\beta}$ if $\bar{\lambda} - \bar{\beta} \in \overset{\circ}{Q}_{+}$ and $\eta_1 \leq \eta_2$ if $\eta_1 - \eta_2 \in Q_+$.
 The roots of $\LL(A) $ is contained in the set $\{ \alpha + n\delta : \alpha \in \overset{\circ}{\Delta}, n \in \ZZ \}$ and we have:
\[
    \LL(A)_{\alpha + n \delta } := 
\begin{cases}
    (\fma_{\alpha} \otimes t^n)(A) \,\,\,\, &\mathrm{if}\,\,\,\, \alpha \neq 0, n \neq 0;\\
   (\fma_{\alpha} \otimes 1)(A)            & \mathrm{if} \,\,\,\,\alpha \neq 0, n = 0;\\
   (\fmh \otimes t^n)(A) \oplus  d_n (A) & \mathrm{if} \,\,\,\, \alpha = 0, n\neq 0;\\
   \bar{\hh}(A) & \mathrm{if}\,\,\,\, \alpha + n \delta = 0 .
\end{cases}
\]
Note that real and imaginary roots of the affine Lie algebra $\gg$ and $\LL(A)$ are same. For a real root $\beta = \alpha + n\delta $, $\beta^{\vee} = \alpha^{\vee} + \frac{2}{\langle \alpha, \alpha \rangle}K$ be the corresponding co-root of $\beta$, where $\alpha^{\vee}$
 is the co-root of $\alpha \in \overset{\circ}{\Delta}$. The Weyl group of $\LL(A)$ is same as that of $\gg$, which is generated by reflections $r_{\beta}$ defined on $\bar{\hh}^*$ by $r_{\beta}(\lambda) = \lambda - \lambda(\beta^{\vee})\beta$,  where $\beta \in \Delta^{\mathrm{re}}$. We denote it by
 $W$.

A representation $V$ of $\LL(A)$ is called a weight module  if 
$V = \bigoplus_{\lambda \in {\bar{\hh}}^*}{V_{\lambda}}$. The subspaces $V_{\lambda} := \{v \in V: h(v)=  \lambda(h)v,\,\, \forall \,\ h \in \bar{\hh} \}$ are called weight spaces of $V$.  We have the following definition. 
\begin{rmk}
For an irreducible $\LL(A)$ weight module $V$, if $\mathrm{Vir}$ acts trivially on $V$, then any element $x \otimes t^{n}, x \in \fma, 0 \neq n \in \ZZ$ also acts trivially on $V$.  For any $v \in V$,  $x \otimes t^{n}v = \frac{1}{n}[d_0, x\otimes t^n] v$ . Now using the bracket operations
one can deduce that $V$ is a  trivial one dimensional $\LL(A)$-module. \end{rmk}
\begin{dfn}
A representation $V$ of $\LL(A)$ is called integrable representation if the following holds:
\begin{enumerate}
\item $V$ is a weight module with finite dimensional weight spaces.
\item Elements $x_\alpha \otimes a$ act locally nilpotently on $V$ for all $x_{\alpha} \in \gg_{\alpha}, \alpha \in \Delta^{re}$ and $ a \in A$, i.e., 
for every $v \in V$ and $x_\alpha \otimes a$, $\alpha \in \Delta^{re}$, $a \in A$, there exists an integer $N = N(\alpha, a, v)$ such that $(x_{\alpha} \otimes a)^N.{v} = 0$.
\end{enumerate}
\end{dfn}

In this paper we classify all the integrable modules of $\LL(A)$ with finite dimensional weight spaces.  First we note down the following which is standard:
\bppsn \label{wp}
Let $V$ be an irreducible integrable module for $\LL(A)$. Then
\begin{enumerate}
\item $P(V) := \{\gamma \in \bar{\hh^*} |\,  V_{\gamma} \neq 0\}$ is $W$- invariant.
\item $\mathrm{dim}V_{\gamma} = \mathrm{dim}V_{w \gamma}, \,\, \forall \,\, w \in W$.
\item If $\lambda \in P(V)$  and $\gamma \in \Delta^{\mathrm{re}}$, then $\lambda (\gamma^{\vee}) \in \ZZ$.
\item If $\lambda \in P(V)$  and $\gamma \in \Delta^{\mathrm{re}}$, and $\lambda(\gamma^{\vee}) >0$, then $\lambda - \gamma \in P(V)$.
\item For $\lambda \in P(V)$, $\lambda(K)$ is an integer independent of $\lambda$.
\end{enumerate}
\eppsn
Unless stated otherwise $V$ will always denote as irreducible integrable $\LL(A)$-module with finite dimensional weight spaces.  By above Proposition affine center will act on $V$ by integer say $m$. Let Vir center
$C$ acts by a scalar $c \in \CC$.

\noindent
{\bf{Case 1):}} $m \neq 0$ and $c \in \CC$.
We will first define the notion of highest weight modules for $\LL(A)$. Let us consider the natural triangular decomposition of $\LL(A) =  \LL(A)^- \oplus \LL(A)^0 \oplus \LL(A)^{+}$, where $\LL(A)^+ : = (\n_+ \oplus \mathrm{Vir}^{+}) (A)$, $\LL(A)^- : = (\n_- \oplus \mathrm{Vir}^{-})  (A)$, $\LL(A)^0 := \bar{\hh} \otimes A$, where $\mathrm{Vir}^+ = \bigoplus_{n > 0}{\CC d_n}$ and $\mathrm{Vir}^- = \bigoplus_{n < 0}{\CC d_n}$.
\begin{dfn}
A $\LL(A)$ weight module is called a highest weight module with highest weight $\lambda$ if there exists a vector $0 \neq v \in V_{\lambda}$ such that $\LL(A)^+ v = 0$ and $U(\LL(A))v = V$.
\end{dfn}
\begin{rmk}
Note that in the above definition the highest weight space $V_{\lambda}$ could be infinite dimensional.
\end{rmk}
Next we construct highest weight $\LL(A)$-modules with one dimensional highest weight space.
Let $\psi: \LL(A)^0 \mapsto \CC$ be a linear map. Let $\CC v_{\psi}$ be a one dimensional representation of $\LL(A)_0$ assigned by $\psi$. Consider the induced Verma module $M(\psi): =  U(\LL(A)^-)\otimes_{\LL(A)^+ \oplus \LL(A)^0} \CC v_{\psi}$, where $\LL(A)^+$ acts trivially on $\CC v_{\psi}$.
By standard argument $M(\psi)$ will have a maximal submodule say $M'(\psi)$ and we have $V(\psi) : = M(\psi)/M'(\psi)$ an irreducible highest weight module for $\LL(A)$ with a highest weight $\lambda = \psi |_{\bar{\hh} \otimes 1}$.  We see that $M(\psi)$ does not have finite dimensional weight spaces whenever $A$ is infinite dimensional, while $V(\psi)$ may have finite dimensional weight spaces depending on $\psi$.
The following proposition encodes the necessary and sufficient condition for $V(\psi)$ to have finite dimensional weight spaces.
\begin{ppsn}
$V(\psi)$ has finite dimensional weight spaces with respect to $\bar{\hh}$ iff $\psi$ factors through $\bar{\hh} \otimes I$, where $I$ is a cofinite ideal of $A$. In this case we have $(\LL \otimes I) V(\psi)= 0$.
\end{ppsn}
\begin{proof}
Let $V(\psi)$ has finite dimensional weight spaces. As in the proof of Proposition 1.1 of \cite{R1} there exists a co-finite ideal $I_1$ of $A$ such that $(\hh' \oplus d_0)(I_1)v=0$, where $v$ denote the highest weight vector. Now consider $I_2= \{a \in A  :   d_{-2}(a)v=0\}$. It is clear that $I_2$  is a co-finite ideal of $A$. To see this for $b \in A$ and $a \in I_2$, consider $[d_0(b), d_{-2} (a)] v = d_0 b \,\, d_{-2}(a) v - \psi(d_0(b)) \,d_{-2}(a) v$. Take $I=I_1I_2$ and consider 
		$$d_2d_{-2}(I)v=d_{-2}(I)d_2v + [d_2 , d_{-2}](I)v={-4}d_{0}(I)v + \frac{1}{2}C(I)v.$$
		Since $I \subseteq I_1$,  $d_{0}(I)v=0$,  we get  $C(I)v=0$, and therefore  $\bar{\mathfrak{h}}(I)v=0$. Remaining part of Proposition follows by same argument as in Proposition 1.1 of \cite{R1}.

\end{proof}
\begin{ppsn}
	Let $V$ is an irreducible integrable module for $\mathcal{L}(A)$. Suppose the central element of affine Lie algebra $K$ acts non-trivially. Then $V$ is a highest weight module or a lowest weight module.
\end{ppsn}
\begin{proof}
We know that $K$ acts by an integer $m \neq 0$. We can assume $m$ is a positive integer while the case when $m$ is negative can be dealt similarly to get a lowest weight module. Let $V = \oplus_ {\lambda \in {\bar{\mathfrak{h}^{*}}}} V_{\lambda}$, dim $V_{\lambda} < \infty$. Now by Theorem 2.4 of \cite{CHA}, there exists $v \in V_{\lambda}$ such that $\mathfrak{n}_{+}v=0$. That means $V_{\lambda + \alpha + n\delta}=0$  for $ \alpha \in \overset{\circ}{\Delta} , n > 0 ,  V_{\lambda + n\delta}= 0$ for $ n>0$, $V_{\lambda + \alpha}=0$ for $ \alpha \in \overset{\circ}{\Delta}_{+}$. In particular $\mathfrak{n}_{+}(a)v=0$ for any $a\in A$ . Since $d_n (a)v \in V_{\lambda + n\delta}$  for any $a\in A$ , therefore for $n>0$,  $d_n (a)v=0$. Therefore $\mathcal{L}(A)^{+} v=0$. Now as $\LL(A)^0$ is an abelian Lie algebra and $V_{\lambda}$ is a finite dimensional irreducible representation of $V_{\lambda}$, hence $\mathrm{dim} \, V_{\lambda} = 1$.
\end{proof}
 We will now classify integrable highest weight modules $V(\psi)$. In other words we have to find necessary and sufficient condition on $\psi$ for that $V(\psi)$ is integrable. First, as $A$ is a finitely generated algebra over $\CC$, it is a Jacobson ring and hence every radical ideal is an intersection of finitely many maximal ideals.
 Let $I$ be a cofinite ideal of $A$ and $\sqrt{I}$ be radical ideal of $I$. Then by above $\sqrt{I} := \{ a \in A: a^n \in I, \, \mathrm{for} \,\, \mathrm{some}\,\, n \} = \M_1 \cap \M_2 \cap \cdots \cap \M_s$. Then it is easy to see that for any Lie algebra $\mathfrak{m}$,  $\mathfrak{m} \otimes \frac{A}{\sqrt{I}} \cong \bigoplus_{i = 1}^{s}{\mathfrak{m}_i}$, where $\mathfrak{m_i}$ is an isomorphic copy of $\mathfrak{m}$, $1 \leq i \leq s$. Recall that $\lambda \in \hh^*$ is called dominant integral if $\lambda(\alpha_i^{\vee}) \in \ZZ_{\geq 0}$ for $ 0 \leq i \leq n$. We have the following definition:
 \begin{dfn}
  Let $I$ be a cofinite ideal of $A$. Let $\psi: \bar{\hh} \otimes A \rightarrow \CC$ be a linear map such that $\psi(\hh' \otimes \sqrt{I}) = 0$. Then $\psi$ is said to be dominant integral (with respect to $\sqrt{I}$) if each $\psi_i$  is dominant integral for $1 \leq i \leq s$, where $\psi_i$ is the restriction of $\psi$ on the $i$-th copy of
  $\bigoplus_{i =1}^{s}{\hh'_i} = \hh' \otimes \frac{A}{\sqrt{I}}$.
 \end{dfn}

   \begin{thm}
   	Let $V(\psi)$ is integrable module for $\mathfrak{L}(A)$. Then there exists a co-finite ideal $I$ of $A$ such that $\psi(d_{0} \otimes I)=0$ , $\psi(C \otimes I)=0$ , $\psi (\fmh \oplus \CC K \otimes \sqrt{I})=0$.
   	
   	\begin{proof}
   		Follows from proposition 2.1 and theorem 2.2 of \cite{R1}.
   		\end{proof}
   \end{thm}
We will now determine the condition on $\psi$ for which $V(\psi)$ is integrable.

 	 \begin{thm}
 Let $V(\psi)$ be an highest weight module for $\LL(A)$ with finite dimensional weight spaces. Assume that there is a co-finite ideal $I$ of $A$ such that
 $\psi(d_{0} \otimes I)=0$ , $\psi(C \otimes I)=0$ , $\psi (\fmh \oplus \CC K \otimes \sqrt{I})=0$ and $\psi$ dominant integral with respect to $\sqrt{I}$. Then $V(\psi)$ is integrable.
 
 \begin{proof}
 Proof is exactly similar to Theorem 2.3 of \cite{R1}. Note that the Virasoro center won't create any complications.
 \end{proof}
\end{thm}

\section{Affine center acting trivially}\label{sec3}
\noindent
{\bf{Case 2):}} $m = 0$ and $c \in \CC$. To make inroads in this case we need to consider different triangular decomposition. First, for any Lie algebra $\mathfrak{m}$, we denote $L[\mathfrak{m}]: = \mathfrak{m} \otimes \CC[t, t^{-1}]$. Let us define
$\LL(A)_{+} = L[\overset{\circ}{\mathfrak{n}}_{+}](A)$, \,$ \LL(A)_{0} = ((L[\fmh]\oplus \CC K )\rtimes \mathrm{Vir}) (A),  \,\,\LL(A)_{-} = L[\overset{\circ}{\mathfrak{n}}_{-}]( A)$.  Now we define class of irreducible integrable $\LL(A)$ modules where the affine center acts trivially.
Let $S$ be a finite dimensional irreducible $\fma$-module. Let $\mathcal{M}$ be a maximal ideal of $A$ and $\psi: A \rightarrow \frac{A}{\mathcal{M}} \cong \CC$ be a algebra homomorphsim. For $\alpha , \beta \in \CC$, define the action of 
$\mathfrak{L}(A)$ on $S \otimes \CC[t, t^{-1}]$ by :
\begin{align*}
(X \otimes t^m)(a)(v \otimes t^n) &=  \psi(a)X(v) \otimes t^{m+n} ;\\
d_n (b)(v \otimes t^m) & = \psi(b)(m + \alpha + \beta + \beta n) v \otimes t^{m+n};\\
K(a) (v \otimes t^{n}) & = 0; \\
C (b)(v \otimes t^{n}) & = 0\,\, ,\,\, a,b \in \CC, X \in \fma.
\end{align*}
It is routine checking to show that with the above action $S \otimes \CC[t, t^{-1}]$ has a $\LL(A)$-module structure and we will denote it as $\LL(S, \alpha, \beta, \psi)$.  In fact it is not hard to see that $\LL(S, \alpha, \beta, \psi)$ is an irreducible $\LL(A)$-module. The following is main result of this section:
\begin{thm}\label{thm}
Let $V$ be an irreducible integrable $\LL(A)$-module with affine center acting trivially on it. Then either $V$ is trivial one dimensional $\LL(A)$-module or $V \cong \LL(S, \alpha, \beta, \psi)$ for some $\alpha, \beta \in \CC$ and $\psi$ as above.
\end{thm}
We need several results to achieve Theorem \ref{thm}. 

Let $M = \{v \in V : x_{\alpha} \otimes  t^{n} v = 0, \forall \,\, x_\alpha \in \fma_{\alpha}, \alpha \in \fpr, n \in \ZZ \}$. Note that $M$ is a highest weight space with respect to new triangular decomposition.
\begin{ppsn} M is non-zero irreducible module for $\LL(A)_0$.
\end{ppsn}
\begin{proof}
Then by result of \cite{CHA} (Theorem 2.4 (ii)), $M$ is nonzero and $\LL(A)_0$ irreducibility follows by PBW theorem. 
\end{proof}
 Now consider the weight space decomposition $\oplus_{\beta \in \bar{\mathfrak{h}}}{M_{\beta}}$.   First note that any $\eta \in P(V)$ can be uniquely written as $ \bar{\eta} +  \eta(d_0) \delta + c \,\omega$, since $\eta(K) = 0$ and $\eta(C) = c$. Now as $\fmh$ lies in center of $\mathfrak{L}{(A)}_0$, it acts as  scalars on $M$. Let $h \in \fmh$ then $h.w = \bar{\lambda}(h)w, \,\,  \forall \,\, w \in M, \mathrm{and} \,\, h \in \fmh$.  Note that $d_0$ leaves each weight space of $M$ invariant, so let $ 0 \neq v \in M_{\beta}$ be such that $ d_0(v) = \xi v, \xi \in \CC$. Now irreducibility of 
$M$ as $\LL(A)_0$ module implies that all the weights of $M$ are contained in the set $\{ \bar{\lambda}  + (\xi + n)\delta + c \,\omega : n \in \ZZ \}$ (see Lemma 2.6 of \cite {R}). It follows from Corollary 3.6 of \cite{KAC} that $\bar{\lambda} \in \overset{\circ}{P}_+$.
 Let $\theta$ be the highest weight of 
$\fma$ and $\theta^{\vee} $ be the corresponding co-root.  Let $L = \Phi (\ZZ[\overset{\circ}{W}(\theta^{\vee})]) $ where $\Phi : \fmh \rightarrow \fmh^*$ is an isomorphism given by non-degenerate bilinear form $\langle \cdot, \cdot \rangle$. Recall that $W = \overset{\circ}{W}  \rtimes T_{L}$, where $\overset{\circ}{W}$ is the Weyl group of $\fma$ and $T_{L}$ is abelian group with elements
$t_{\alpha} \in \mathrm{GL}{(\bar{\mathfrak{h})}^*}$, $\alpha \in L$ defined by $t_{\alpha}(y) = y +   y(K) \alpha - (\langle y, \alpha \rangle + \frac{1}{2}|\alpha|^2  y(K)) \delta$ (see \cite{KAC}). As $K$ acts trivially we have
$t_{\alpha}(y) = y - \langle y, \alpha \rangle \delta ,$ and so we have $t_{\alpha + \beta} = t_{\alpha} t_{\beta}$. Let $s_{\lambda} := \mathrm{min} \{\lambda(h) : \lambda(h) >0, h \in \ZZ[\overset{\circ}{W}(\theta^{\vee})] \}$. We have the following from \cite{CG}(Lemma 3.1).
\begin{ppsn} \label{vcp}
Let $r \in \ZZ$ be such that  $M_{\bar{\lambda }+ (\xi + r ) \delta + c \,\omega} \neq 0$. Then there exists $w \in W$ such that $w(\bar{\lambda} +( \xi + r) \delta + c \, \omega) = \bar{\lambda} + (\xi + {l} )\delta + c \, \omega,  \, 0 \leq {l} < s_{\lambda}$.
\end{ppsn}
Using above Proposition we have the following:
\begin{crlre}
There are only finitely many $\gg(A)$-submodules of $V$. In particular there exists finitely many weight spaces $M_{\beta_i}$, $1 \leq i \leq l$ such that
 $$M = \sum_{i = 1}^{l}{U((L[\fmh] \oplus \CC K)(A)) M_{\beta_i}} \,\,\,(\text{sum need not be direct}). $$
\end{crlre}
\begin{proof}
Define $\tilde{M} =\displaystyle{ \bigoplus_{l = 0}^{s_{\lambda} -1}{M_{\bar{\lambda} + (\xi + {l} )\delta + c \, \omega}}}$. Notice that any $\gg(A)$ submodule of $V$  is generated by the vectors of $\tilde{M}$. Hence there only finitely many 
$\gg(A)$ submodules of $V$ and consequently we have $V = \displaystyle{\sum_{i = 0}^{s_{\lambda} -1}{U(\gg(A))M_{\beta_i}}}$(sum need not be direct), where $\beta_i = \bar{\lambda} + (\xi + {i} )\delta + c \, \omega$. From this we deduce that 
$M = \displaystyle{\sum_{i = 0}^{s_{\lambda} -1}{U((L[\fmh] \oplus \CC K)(A)) M_{\beta_i}}} $.
\end{proof}

\begin{ppsn}
There exist a positive integer $p$ such that $P(V) \subseteq \{ \bar{\gamma_i }+ (\xi + n)\delta + c \,\omega: 1 \leq i \leq p, n \in \ZZ \} $.
\end{ppsn}
\begin{proof}
We have seen that set of all weight of $M$ are contained in $\{ \bar{\lambda}  + (\xi + n)\delta + c \,\omega : n \in \ZZ \}$. As $V = U(L[\overset{\circ}{\mathfrak{n}}_{-}] (A))M$, we note that weights of $V$ are contained in the set
$\{ \bar{\lambda} - \eta + (\xi +n) \delta + c \, \omega: \eta \in \overset{\circ}{Q}_{+}, n \in \ZZ\}$.  Let $\bar{\lambda} - \eta_i + (\xi +m) \delta + c\, \omega $ be any weight of $V$, then by Proposition \ref{wp} (i) and Lemma A of section 13.2 of  \cite{HUM}, we get that the weight $\bar{\lambda} - \eta_i + (\xi +m) \delta + c\, \omega$ is $\overset{\circ}{W}$-conjugate to a weight $\bar{\beta_i} + (\xi +m) \delta + c\, \omega$, where $\bar{\beta_i} \in \overset{\circ}{P}_{+}$ and $\bar{\beta_i} \leq_{0} \bar{\lambda}$. But by Lemma B of section 13.2 of \cite{HUM}, there are finitely many dominant weights $\leq_{0}\bar{\lambda}$, say $\{\bar{\gamma_1} = \bar{\lambda}, \ldots, \bar{\gamma_d}\}$. Finally, listing all $\overset{\circ}{W}$-conjugates, $\overset{\circ}{W}\{\bar{\gamma_1}, \ldots, \bar{\gamma_d}\} =  \{ \bar{\gamma_1}, \ldots, \bar{\gamma_p} \}$, we get the desired result.
\end{proof}
 
\begin{ppsn} \label{ps1}
Let $V$ be an irreducible integrable $\LL(A)$-module with affine center acting trivially on it. Then its weight spaces are uniformly bounded. Moreover considered as $d_0$ weight module, the weight spaces of $V$ are also uniformly bounded.
\end{ppsn}
\begin{proof}
Uniform bound for $\bar{\hh}$ weight spaces $M$ follows by Proposition \ref{vcp}. But as $P(V) \subseteq  \{ \bar{\gamma_i }+ (\xi + n)\delta + c \,\omega: 1 \leq i \leq p, n \in \ZZ \}$, a repetitive application of Proposition \ref{vcp} gives uniform bound for $\bar{\hh}$ weight spaces of $V$. To obtain an uniform bound on $d_0$ weight spaces let $0 \neq v \in V$ such that $d_0 v = \zeta v $. Then as  $[d_0, d_n(a)] = n d_n(a)$ and $[d_0, (x \otimes t^m)(b)] = m (x \otimes t^m)(b)$ for $a, b \in A, x \in \fma$, we note that every element of $U(\LL(A))$ decomposes into finitely many eigenvectors of $d_0$ with integer eigenvalues.   It follows that with respect to $d_0$, $V = U(\LL(A))v$ has the weight space decomposition $V = \bigoplus_{n \in \ZZ}{V_{\zeta +n}}$. 
Now as we have seen set of all $\bar{\hh}$ weights of $V$ are uniformly bounded by $N : = \mathrm{Max}\, \{\mathrm{dim} \, V_{\bar{\gamma_i} + (\zeta + r_i ^{j}) \delta + c \, \omega} : 1 \leq i  \leq p, 0 \leq  r_i  ^{j}< s_{\gamma_i}\}$, we deduce that dimensions of $d_0$ weight spaces is bounded by $p N$, i.e.,
$\mathrm{dim} V_{\zeta +n} \leq p N , \forall n \in \ZZ$.
 \end{proof}
 \noindent
 We write couple of remarks here:
 \begin{rmk}
 We won't need Proposition \ref{ps1} in full force for this classification problem. In fact what we really need is uniform bound on  $d_0$ weight spaces of $M$ which is clear.  Now from Theorem II(7) of \cite{CHPI}  it follows that Vir center $C$ acts trivially on $M$ and hence on $V$.
 \end{rmk}
 
 \begin{rmk}
 In the case when $\bar{\lambda} = 0$, by weight argument it follows that $L[\overset{\circ}{\mathfrak{n}}_{-}](A)M = 0 $ and using bracket relations we deduce that $L[\fmh](A)M = 0$. Hence we have $V = M$ is an irreducible $\mathrm{Vir}(A)$-module and appealing to Theorem 4.7 of \cite{SAV} 
 proves our Theorem \ref{thm}.
 \end{rmk}
 Let $\hh_1$ be a finite dimensional vector space with non-degenerate form $(,)$. Consider the standard Heisenberg algebra $L(\hh_1) \oplus \CC K_1$. We need the following result due to Futorny \cite{F} (Proposition 4.3 (i) and 4.5):
\begin{ppsn}\label{pf1}
Let $\bar{V}$ be $\ZZ$-graded $L(\hh_1) \oplus \CC K_1$ module with finite dimensional graded spaces. Suppose that $K_1$ acts as a non zero scalar on $\bar{V}$. Then $V$ contains a highest weight module or a lowest weight module. In particular dimensions of
graded spaces of $\bar{V}$ are not uniformly bounded.
\end{ppsn}
\begin{ppsn}\label{p1}
$K (A)$ acts trivially on $V$. 
\end{ppsn}
\begin{proof}
As we have seen that $M$ has uniformly bounded weight spaces, and all the weights of $M$ are of the form $\bar{\lambda} + (\alpha +n) \delta + c \, \omega$. So by letting $M_{\bar{\lambda}+ (\alpha + n) \delta + c \, \omega} : = M_n$, we see that $M$ is $\ZZ$-graded $L(\hh) \oplus \CC K$-module. Let fix $a \in A$ 
and let $h_1, h_2 \in \hh$ such that $\langle h_1, h_2 \rangle \neq 0$. Consider the Heisenberg subalgebra $\mathfrak{H}(a)$ spanned by the elements $\{h_1 \otimes t^n, h_2 \otimes t^m(a):= h_2 \otimes t^m \otimes a, K(a): n>0, m<0, n,m \in \ZZ\}$ with the Lie bracket:
$[h_1 \otimes t^n, h_2 \otimes t^m(a)] = \langle h_1,  h_2 \rangle \delta_{n, -m}K(a)$, and $K(a)$ acts as central element.
Then $M$ is 
$\ZZ$-graded $\mathfrak{H}(a)$-module with uniformly bounded graded spaces. Now Proposition \ref{pf1} forces $K(a)$ to act trivially on $M$. As $a \in A$ was arbitrary, we see that $K(A)$ acts trivially on M and hence on $V$ as $K (A)$ lies in 
center of $\LL(A)$ and $V$ is an irreducible $\LL(A)$-module.
\end{proof}

We next aim to prove existence of a cofinite ideal $R$ of $A$ such that $\LL(R) $ acts trivially on $V$. First step in this direction is the following  proposition:
\begin{ppsn}\label{p4}
There exist a cofinite ideal $I$ of $A$ such that $\mathfrak{g}'(I) V = 0.$
\end{ppsn}

\begin{proof}
Let $\alpha_1, \alpha_2, \ldots, \alpha_n$ denote simple roots of $\fma$. Let $x_{\alpha_i} \in \fma_{\alpha_{i}}$ such that $[x_{\alpha_i}, x_{-\alpha_i} ] = h_{\alpha_i}$. Let $M = \bigoplus_{\beta \in \bar{\hh}^*}{M_{\beta}}$.  Take $\eta \in \bar{\hh}^*$ such that
$M_{\eta} \neq 0$ and let $w_1, \ldots, w_r$ be a basis for $M_{\eta}$.  Define for $1 \leq i \leq n, 1 \leq j \leq r$,  $I_{i, j} := \{ a \in A : (x_{- \alpha_{i}}\otimes 1)(a) w_j = 0\}$.  Observe that $I_{i,j}$'s are cofinite ideal of $A$ as $\fmh(A)$ acts as scalars on $M$. 
Define $I = \prod_{i,j}{I_{i,j}}$. Then $I$ is a cofinite ideal and we have $(x_{-\alpha_i}\otimes 1)( I )M_{\eta} = 0$ for all $1\leq i \leq n$. It follows that $(\fmh \otimes 1)( I) M_{\eta} = 0$. Also for any $0 \neq m \in \ZZ$ and $1 \leq i \leq n,$
$$(h_{\alpha_i} \otimes t^m )(I) M_{\eta} = [(x_{\alpha_i} \otimes t^{m}) (1), (x_{-\alpha_i}\otimes 1) (I) ] M_{\eta} $$ $$=  (x_{\alpha_i} \otimes t^{m})(1) \,(x_{-\alpha_i}\otimes 1) (I) M_{\eta} - (x_{-\alpha_i}\otimes 1)(I) \,( x_{\alpha_i} \otimes t^{m})(1) M_{\eta}.$$
Note that first term of RHS is zero by above argument and second term is also zero as $M_{\eta} \subseteq M$. Now we will prove that $L[\overset{\circ}{\mathfrak{n}}_{-}](I) M_{\eta} = 0$. We will first prove for any $m \in \ZZ,$  $1 \leq i \leq n$, $(x_{-\alpha_i} \otimes t^{m})(I)M_{\eta}$
is a set of  highest weight vectors of $V$.  For any $\beta \in \fpr$, $m, n \in \ZZ$, $a \in A$ consider $$(x_{\beta} \otimes t^{n}) (a) . (x_{- \alpha_i} \otimes t^{m}) (I) M_{\eta} = x_{- \alpha_i} \otimes t^{m} (I) \, (x_{\beta} \otimes t^{n})( a)M_{\eta} $$ $$+ (x_{\beta - \alpha_i} \otimes t^{m +m}) (I) + n\delta_{n , -m}\langle x , y \rangle K(I).$$ We see that second term of RHS is zero as  one three possibilities can occur (i)$\beta - \alpha_i \notin \overset{\circ}{\Delta}$; (ii) $\beta - \alpha_i \in \fpr$; (iii) $\ \beta = \alpha_i$. The first and third sum of RHS are easily seen to be zero. Now we prove that 
for $\beta \in \fpr$, $(x_{-\beta} \otimes t^{m})(I)M_{\eta}$ is a set for highest weight vectors of $V$.  We will prove this using induction on height $\beta$. For height equal to 1 we are done.  Let height of $\beta > 1$, for any $\gamma \in \fpr$ consider 
$$(x_{\gamma} \otimes t^{n}) (a) .\, (x_{- \beta} \otimes t^{m}) (I)M_{\eta} = (x_{- \beta} \otimes t^{m}) (I) \,(x_{\gamma} \otimes t^{n})(a)M_\eta$$ $$ + (x_{- \beta + \gamma} \otimes t^{m +n})(aI) M_{\eta} + \langle x_\gamma , x_{-\beta} \rangle n \delta_{n, -m}K(aI) M_{\eta}.$$ Now clearly first and third terms are zero and second term is zero by induction. Hence we have proved that $L[\overset{\circ}{\mathfrak{n}}_{-}] (I) M_{\eta}$ is a set of highest weight vectors of $V$ which are not contained in $M$ hence must be zero.  Putting everything together, we have $\gg'(I) M_{\eta} = 0$. But as $\gg' (I)$ is
an ideal of $\LL(A)$, the set $\{v \in V: \gg' (I) (v) = 0 \}$ , which contains $M_{\eta}$ is a non-zero $\LL(A)$ submodule of $V$ hence must be $V$.
\end{proof}
\begin{ppsn}\label{p3}
There exists a cofinite ideal $R$ of $A$ such that $\mathrm{Vir}(R) M = 0$.
\end{ppsn}
\begin{proof}
 Let $M = \bigoplus_{n \in \ZZ}{M_{\xi + n}}$ be weight decomposition of $M$ with respect to $d_0$.  Let $i \in \ZZ$ such that 
$M_{\xi + i} \neq 0$ and for $j \neq 0$, consider the set $I_j = \{a \in A : d_j (a) M_{\xi + i} = 0\}$. Then $I_j$ is an ideal of $A$ as for any $b \in A$ and $a \in I_j$ consider $[d_0(b), d_j(a)] M_{\xi + i} = d_0(b) d_j(a)M_{\xi + i} -  d_j(a) d_0(b)M_{\xi + i} $. Now first term of RHS is zero and for the second term note that as $d_0(b)$ preserves $M_{\xi + i}$, it should be also zero. Also $I_j$ is a cofinite ideal as it can be seen as a kernel of a linear map $\phi_j: A \rightarrow M_{\xi + i}  \oplus \cdots \oplus M_{\xi + i} (q \,\, \mathrm{times})$, $\phi_j(a) = (d_j (a) w_1,\ldots, d_j(a) w_{q} )$, where $\mathrm{dim}\, M_{\xi + i}
=q $ and $\{w_1, \ldots ,w_q\}$ is a basis of $M_{\xi + i} $. Similarly for $-j$  we get an ideal $I_{-j}$ and consider $$[d_j(I_j), d_{-j}(I_{-j})] = (-2)d_0(I_{j} I_{-j}) + \frac{j^3 -j}{12} C(I_{j} I_{-j}).$$  Similarly replacing $j$ by $j+r$ where $j+r \neq 0$ we get
$$[d_{j+r}(I_{j+r}), d_{-(j + r)}(I_{-(j + r)})] =$$ $$ (-2(j+r))d_0(I_{j+r} I_{-(j+r})) + \frac{(j+r)^3 -(j+r)}{12} C(I_{j+r} I_{-(j+r)}).$$ Define a cofinite ideal $S = I_{j} I_{-j} I_{j+r}I_{-(j+r)}$ we deduce that $d_0(S)M_{\xi + i} = 0  = C(S)M_{\xi + i}$. Now as $M = \sum_{i = 1}^{m}{U(L[\fmh](A))M_{\beta_i}}$, where $\beta_i
= \bar{\lambda} + (\xi + r_i)\delta + c \, \omega, r_i \in \ZZ$, and noting that as a  $d_0$ weight space $M_{\beta_i} = M_{\xi + r_i}$ and replacing the cofinite ideal S by  product  of finitely many cofinite ideals  say $S'$, we get $d_0(S')(\sum_{i}^{m}M_{\xi + r_i}) = 0$ and $C(S')(\sum_{i}^{m}M_{\xi + r_i}) $.
Now let $R := S'I$, where $I$ is the cofinite ideal of Proposition \ref{p4},
we have the following claim:\\
{\bf{Claim:} }$d_0 (R) (M) = d_0(R)\sum_{i = 1}^{m}{U(L[\fmh](A))M_{\beta_i}} = 0$.\\
{\bf{Proof of the Claim:}} It is enough to show that $[d_0(R), U(L[\fmh](A))]M_{\beta_i} = 0$, $ 1 \leq i \leq m$. Let $U(L[\fmh](A))  = \CC \oplus U_{0}(L[\fmh](A))$, where $U_{0}(L[\fmh](A)):= L[\fmh](A)U(L[\fmh](A))$ is the augumentation ideal of $U(L[\fmh](A))$. It is suffices to show
that $[d_0(R), U_{0}(L[\fmh](A))] M_{\beta_i} = 0$ for a fixed $i$, $1 \leq i \leq m$. 
An arbitrary element of $U_{0}(L[\fmh](A))$ is sum of elements of the form $h_{i_1} \otimes t^{m_1}(a_{j_1}) \cdots h_{j_t} \otimes t^{m_t}(a_{j_t})$, where $h_{i_k} \in \fmh, 1 \leq k \leq t$, $m_i \in \ZZ, ,1 \leq i \leq t, a_{j_k} \in A, 1 \leq k \leq t$.
We will prove $$[d_0 (R), h_{i_1} \otimes t^{m_1}(a_{j_1}) \cdots h_{i_t} \otimes t^{m_t}(a_{j_t})]M_{\beta_i} =  0.$$ We use induction on $t$. For $t = 1$, $b \in R$  consider $$[d_0(b), h_{i_1} \otimes t^{m_1}(a_{j_1})]M_{\beta_i} = h_{i_1} \otimes t^{m_1}(a_{j_1})d_0(b) M_{\beta_i} + 
m_1 h_{i_1} \otimes t^{m_1}(ba_{j_1})M_{\beta_i}$$
which is clearly zero. Now for arbitrary $t$ consider
$$[d_0 (b), h_{i_1} \otimes t^{m_1}(a_{j_1}) \cdots h_{i_t} \otimes t^{m_t}(a_{j_t})]M_{\beta_i}  = $$
$$ \,\,\,\,\,\,\,\,\,\,\,\,\,\,\, d_0 (b) h_{i_1} \otimes t^{m_1}(a_{j_1}) \cdots h_{i_t} \otimes t^{m_t}(a_{j_t})M_{\beta_i} =$$
$$ h_{i_1} \otimes t^{m_1}(a_{j_1}) d_0(b) \cdots h_{i_t} \otimes t^{m_t}(a_{j_t})M_{\beta_i} + m_1h_{i_1} \otimes t^{m_1}(ba_{j_1}) \cdots h_{i_t} \otimes t^{m_t}(a_{j_t})M_{\beta_i}.$$

First term of RHS is zero by induction and and also is the second term as $U(L[\fmh](A))$ is an abelian algebra. This completes the proof of claim. Finally,
  $0 \neq n \in \ZZ, \frac{1}{n}d_n (R)(M) = [d_0(R), d_n(1)]M = 0$ as $d_n(1)$ leaves $M$ invariant. 
 \end{proof}
 Let $R$ be an ideal from the above Proposition. We have the following: 
 \begin{ppsn} \label{mp1}
 $\mathfrak{L}(R) V = 0$. 
 \end{ppsn}
\begin{proof}
 We already have $\gg'(R)V = 0$, and also $\text{Vir}(R)M = 0$. So the set $\mathcal{W}: = \{v \in V: \mathcal{L}(R)v = 0\}$ is non-zero as it contains $M$.
 We will be done if we show that $\mathcal{W}$ is a $\LL(A)$ submodule. But this is a easy checking. For $x, y \in \fma, n,m \in \ZZ, a \in A$, and $b\in R$ and $v \in \mathcal{W}$ consider
 $y\otimes t^{n}(b) x \otimes t^{m}(a)v = x \otimes t^{m}(a)y\otimes t^{n}(b) v + [y\otimes t^{n}(b), x\otimes t^{m}(a)]v$ which is clearly zero. Similarly one can check other cases also.
\end{proof}
Now as $A/R$ is finite dimensional, hence both Artinian and Noetherian, without loss of generality we can assume
that $R = \M_1^{k_1}\M_2^{k_2}\cdots \M_q^{k_q}$, where $\M_i$ for $1 \leq i \leq q$ denotes maximal ideal of $A$.
We have the following proposition:
\begin{ppsn}
Let $R_1$ and $R_2$ be two co-prime  cofinite ideals of $A$ such that $\mathfrak{L}_0(R_1 R_2) M = 0$, then either $\mathfrak{L}_0(R_1)M = 0$ or $\mathfrak{L}_0(R_2)M = 0$.
\end{ppsn}
\begin{proof}
Let $M = \bigoplus_{n \in \ZZ}{M_{\xi + n}} $ be $d_0$ weight decomposition of $M$. Now as $R_1$ and $R_2$ are coprime $\LL{_0}^{1,2} : = \LL_{0}(A)/ \LL_{0}(R_1 R_2) \cong \LL_{0}(A/R_1) \oplus \LL_{0}(A/R_2)$.  Denote $\LL_0^{1} :=  \LL(A/R_1) $ and $\LL_0^{2} :=  \LL(A/R_1)$.
Let $d_0^{1} = d_0 \otimes (1 + R_1) \in \LL_0^{1} $ and $d_0^{2} = d_0 \otimes (1 + R_2) \in \LL_{0}^2$ and $d_0^{1,2} = d_0^{1} + d_0^{2}$. We need to show that either $\LL_0^{1}$ or $\LL_0^{2}$ act trivially on $V$.  Now as $d_0^{1,2}, d_0^{1}, d_0^2$ commute with each other 
by similar argument as Theorem 4.3 of \cite{ZRX}, $d_0^{1}$ and  $d_0^2$ are diagonalisable on $V$. Now as $[d_0^{1,2}, (h \otimes t^n) \otimes(a + R_1 R_2)] = n(h\otimes t^n) \otimes (a + R_1 R_2 )$, where $h \in \fmh$,
 we can write $M = \bigoplus_{p.q \in \ZZ}{M_{p,q}}$ where $M_{p, q} = \{v \in M: d_0^1 (v) = (\alpha_1 + p)v, d_0^2 (v) = (\alpha_2 + q) v\}$. Here $\alpha_1 + \alpha_2 = \alpha$.  The subspaces 
 $M^{(q)} := \bigoplus_{p\in \ZZ}{M_{p,q}}$ and $M_{(p)}: =  \bigoplus_{q \in \ZZ}{M_{p,q}}$ are $\LL_{0}^1$ and $\LL_{0}^2$ modules respectively. Now if $\LL_{0}^1$ acts trivially on $M_{(q)}$, then since $\LL_0^{1}$ and $\LL_0^{2}$ commute with each other, it would force that
 $\LL_0^{1}$ acts trivially on $M = U(\LL_0^{2})M_{(q)}$. Hence assuming $\LL_0^{1}$ and $\LL_0^{2}$ acts non-trivially on $M_{(q)}$ and $M^{(p)}$ respectively.  Now for $M_{q} \neq 0$, as a copy of $\mathrm{Vir}$ is contained in $\LL_0^{1}$, it follows that $M_{p, q} \neq 0$ for all
 $\alpha_1 + p \neq 0$ and similarly for $M^{(p)} \neq 0$ for all $\alpha_2 + q \neq 0$.  But this forces that $\bigoplus_{p+q = 0}{M_{p,q}} = M_{\alpha_1 + \alpha_2} = M_{\alpha}$ is infinite dimensional which is a contradiction. 
 \end{proof}
In the light of the above proposition we can assume that $\mathfrak{L}_0({\M_1}^{k_1})M = 0$.  We have the following proposition:
\begin{ppsn}\label{p5}
Let $\M_1$ be a maximal ideal such that $\mathfrak{L}_0({\M_1}^{k})M = 0$ for some positive integer $k$, then we have $\LL_0(\M_1) M = 0$.
\end{ppsn}
\begin{proof}
We use induction on $k$. For $k = 1$ result is obvious. Assume $k = 2 $.
Let us denote by $G:= L[\fmh] \rtimes \mathrm{Vir}$. It follows that $G(M)$ is an abelian subalgebra of $G(A)$. It is enough  to show that $G(\M_1)M = 0$ as $K(A)$ acts trivially on $V$. Now imitating step by step from Proposition 4.1 of \cite{ZRX} or Proposition 4.5 of \cite{SAV}, we obtain
that for any $b \in \M_1$, $d_0(b)$ acts trivially on $M$. Hence for $0 \neq m \in \ZZ, h \in \fmh, e \in \M_1, \frac{1}{m}(h \otimes t^m (e))M = [d_0(e), h \otimes t^{m}(1)]M = 0$ and similarly  $d_k(e)M = 0$ for $0 \neq k \in \ZZ, e \in \M_1$.  So we only need to prove that
$(\fmh \otimes 1) (\M_1)M = 0$. Note that $(\fmh \otimes 1) (\M_1)$ lies in center of $G$.  So we have $h\otimes1(e) v = R_{h,e} \,v$, where $h \in \fmh, 0 \neq v \in M, e\in \M_1$ and $R_{h, e}$ a scalar depending on $h$ and $e$. We aim to show that $R_{h, e} = 0 \,\,\forall  \,\,h \in \fmh, e \in \M_1$.
It is enough to show that $R_{h_i, e} = 0$, for $1 \leq i \leq n$, where $h_i = [ x_{\alpha_i}, x_{- \alpha_i}], x_{\alpha_i} \in \fma_{\alpha_i}$.  We first make the following:\\
{\bf{Claim:}} $\LL(\M_1^2)V = 0$. Assuming the claim we prove that $R_{h_i, e} = 0$, for $1 \leq i \leq n$. Let fix $j$ from $1 \leq j \leq n$, $0 \neq v \in M$ and $e \in \M_1$. Let $k$ be the smallest positive integer such that $((x_{-\alpha_j}\otimes1)(e))^k v  =0. $ Here we use that $V$ is an integrable module.
Now consider 
$$0 = (x_{\alpha_j} \otimes 1)(1) ((x_{-\alpha_j} \otimes 1) (e))^k v = k ((x_{-\alpha_j} \otimes )1(e))^{k-1}  (h_j \otimes 1)(e)v$$ $$  = k R_{h_j, e} ((x_{-\alpha_j} \otimes 1) (e))^{k-1} v .$$
Above we used the claim so that $(h_j \otimes 1)(e)$ commutes with $((x_{-\alpha_j} \otimes 1)(e))^{k-1}$.  Now we deduce that $R_{h_j, e} = 0$.\\
{\bf{Proof of the claim:}} First as $\LL_0(\M_1^2) M = 0$, by similar arguments in Proposition \ref{p4} $L[\overset{\circ}{\mathfrak{n}}_{-}](\M_1^2)M$ is a set highest wight vectors of $V$ such that $L[\overset{\circ}{\mathfrak{n}}_{-}](\M_1^2)M \,\,\cap \,\,M = 0$. Hence $L[\overset{\circ}{\mathfrak{n}}_{-}](\M_1^2)M =0$
and we have $\LL(\M_1^2)M = 0$. Finally $\LL(\M_1^2) V = \LL(\M_1^2) U(L[\overset{\circ}{\mathfrak{n}}_{-}](A))M$. But as $\LL(\M_1^2)M = 0$, to show $ \LL(\M_1^2) V = 0$, we need to show that $[\LL(\M_1^2), U_0(L[\overset{\circ}{\mathfrak{n}}_{-}](A))]M = 0$.  This can be seen using an simple induction
on the length of typical vectors of  $U_0(L[\overset{\circ}{\mathfrak{n}}_{-}](A)).$

Now we prove the result for general $k$. Define the subalgebra $\LL_0^{k-1} :=  (L[\fmh] \rtimes \mathrm{Vir} )( \frac{{\M_1}^{k-1}}{{\M_1}^{k}}) $ of the Lie algebra $\LL_0^{k}: = (L[\fmh] \rtimes \mathrm{Vir})( \frac{A}{{\M_1}^{k}})$.  Again $\LL_0^{k-1}$ is an abelian subalgebra of
$\LL_0^{k}$. Here we are again as in the previous case as $(L[\fmh] \rtimes \mathrm{Vir} )( (\frac{{\M_1}^{k-1}}{{\M_1}^{k}})^2) M = 0. $ Hence we have $\LL_0^{k-1} M = 0$ and hence $G(\M_1) M = 0$ by induction.  By similar argument as we get $L[\overset{\circ}{\mathfrak{n}}_{-}] (\M_1)(M) =0$, hence $\mathfrak{L}(\M_1)M = 0$. \end{proof}
\begin{ppsn} \label{p6}
$\LL(\M_1)V = 0$.
\end{ppsn}
\begin{proof}
Proof follows in the similar lines of the claim of above Proposition.

\end{proof}
Now we are in a position to prove Theorem \ref{thm}. By Proposition \ref{p6} it follows that $V$ is an irreducible module for $\gg' \rtimes \mathrm{Vir} = \fma \otimes \CC[t,t^{-1}] \otimes \CC K \rtimes \mathrm{Vir} $. Now we in similar state of 
Section 4 \cite{RJ}, and appealing to Theorem 4.3 of \cite{RJ}, completes our proof of Theorem \ref{thm}.
\section{Affine Central Operators} \label{Sec4}
In this section we recall results from Section 3 of \cite{R1} on affine central operators. All the results of Section 3 of \cite{R1} will go through in our setting except last Proposition 3.8 where Virasoro center makes appearance. Recall that
present paper deals with independent affine and Vir centers while in \cite{R1}, they had taken to be  same. We recall some notations:
\begin{dfn}

A module $V$ of $\LL(A)$ is said to be of category $\mathcal{O}_c$ if the following holds:
\begin{enumerate}
\item Virasoro center $C$ acts by a fixed scalar $c$ on $V$.
\item $V$ is weight module for $\LL(A)$ with respect to Cartan subalgebra $\bar{\hh}$ and has finite dimensional weight spaces.
\item For any $v \in V$ and $a \in A$ we have $X_{\alpha}(a)v = 0$ for $\mathrm{ht} \alpha > > 0, \alpha \in \Delta_{+}$ and $X_{\alpha} \in \LL_{\alpha}$.
\end{enumerate}
\end{dfn}
\noindent
{\bf{Note:}} Since we are assuming that $C$ acts by $c$, we will work with the Cartan subalgebra $\hh$ of $\gg$, i.e., $\fmh \oplus \CC K \oplus \CC d_0$.

We recall some well known facts from \cite{KAC}. Let $\alpha_0 = - \beta + \delta$, where $\beta$ is a highest root of $\fma$ and $\delta$ is the standard null root of $\fma$.
Let $\rho \in \hh^*$ such that $\langle \rho, \alpha_i \rangle = \frac{1}{2} \langle \alpha_i, \alpha_i \rangle$ for all $\alpha_i, 0 \leq i \leq n$. Let $\bar{\rho} = \rho|_{\fmh}$. Recall that
$\delta(\fmh) = 0, \delta(K) = 1, \delta(d_0) = 1$. Then $\rho = \bar{\rho} + h^{\vee}\Lambda_0$ (see \cite{KAC} 6.2.8) where $h^{\vee}$ is the dual Coxeter number of $\fma$. 
Let $\Delta = \{ \alpha + n \delta, m\delta: 0 \neq m \in \ZZ, n \in \ZZ, \alpha \in \overset{\circ}{\Delta}\}$. Let $ \{ h_i : 1 \leq i \leq n \} $ be a standard basis of $\fmh$. Let $ \{ h^i : 1 \leq i \leq n \}$ be the dual basis of $\{ h_i \}$, i.e.,
$\langle h_i , h^{j} \rangle = \delta_{i, j}$ for $1 \leq i,j \leq n $. Then $\{ h_i, K, d_0: 1 \leq i \leq n\}$ and $\{ h^i,  d_0, K: 1 \leq i \leq n\}$ are dual basis of $\hh$. Let $x_{\alpha} \in \fma_{\alpha}$ and $x_{- \alpha} \in \fma_{- \alpha}$ such that 
$\langle x_{\alpha}, x_{- \alpha} \rangle = 1$ so that $[x_{\alpha}, x_{- \alpha}] = \gamma^{-1}(\alpha)$ (see \cite{KAC} Theorem 2.2 (e)). Then the Casimir operator on $\gg$ is defined by 
$$\Omega = 2 \gamma^{-1}(\rho) + \sum{h_i h^i} + 2 K d_0 + 2 \sum_{\alpha \in \overset{\circ}{\Delta}} \sum_{n > 0}{x_{-\alpha} \otimes t^{-n} x_{ \alpha} \otimes t^n}$$ $$ + 2\sum_{n > 0} \sum_{i}{h_i \otimes t^{-n} h^i \otimes t^n} +
2 \sum_{\alpha \in \overset{\circ}{\Delta}_+}{x_{-\alpha} x_{ \alpha}}.$$
Define for $a, b \in A$
$$\Omega_{a, b}^{1} = \sum_{\alpha \in \overset{\circ}{\Delta}} \sum_{n > 0}{x_{-\alpha} \otimes t^{-n}(a) x_{ \alpha} \otimes t^n}(b),$$
$$\Omega_{a, b}^{2} =  \sum_{i} \sum_{n > 0}{h_i \otimes t^{-n}(a) h^i \otimes t^n}(b),$$
$$\Omega_{a, b}^{3} = \sum_{\alpha \in \overset{\circ}{\Delta}_+}{x_{-\alpha}(a) x_{\alpha}(b)} .$$
Then $$ \Omega(a, b) = 2 \gamma^{-1}(\rho)(ab) + \sum{h_i (a)h^i(b)} +  K(a) d_0(b) + K(b) d_0(a) $$ $$+ \Omega_{a, b}^{1} + \Omega_{b, a}^{1} + \Omega_{a, b}^{2} + \Omega_{b, a}^{2} + \Omega_{a, b}^{3} + \Omega_{b, a}^{3} .$$
\begin{ppsn} [Propsition 3.1 of \cite{R1}]
$[\Omega(a,b), \gg'] = 0$ on objects of $\mathcal{O}_c$.
\end{ppsn}
\begin{dfn}
An operator $Z$ on objects of $\mathcal{O}_c$ is called an affine central operator if $Z$ commutes with an action of $\gg'$.
\end{dfn}
For example $\Omega(a,b)$ is an affine cental operator.
\subsection{}
Let for $j \neq 0$, $T_{j}(a,b) = - \frac{1}{j}[d_j, \Omega(a,b)]$. It is easy checking that $T_j (a, b)$ is an affine central
operator. All the result of Section 3 of \cite{R1} will go through. In particular $T_j (a, b)$ can be written down explicitly for the
purpose if application. We now rewrite Proposition 3.8 of \cite{R1} where Virasoro center will appear. Recall that $C$ denotes Virasoro center and $C(a): = C \otimes a$ for $a \in A$.
\begin{ppsn}
For $ 0 \neq j \in \ZZ, k \in \ZZ$ $$[d_k, T_j (a, b) ] = (j - k)T_j (a,b) - \delta_{j +k, 0} \frac{k^3 - k}{6} \mathrm{dim}\, \fma \,\, K(ab) $$ $$+ \delta_{j+k, 0} \frac{k^3 - k}{12}\big(K(a) C(b) + K(b)C(a) + 2 h^{\vee} C(ab) \big).$$
\end{ppsn}
\begin{rmk}
Suppose $V(\lambda)$ is an irreducible highest weight module for affine Lie algebra $\gg$ at non-critical level $(K + h^{\vee} \neq 0)$. Then it is well known that Virasoro algebra act
on $V(\lambda)$ by famous Sugawara operators (See \cite{KAC}, Chapter 12). In the present setup of this paper Affine and Virasoro centers differ. In this case (we take $A = \CC$) our 
affine central operators $T_j (1,1)$ will collapse to zero. But where the algebra $A$ is non-trivial we get some interesting affine central operators as explained in \cite{R1}.
\end{rmk}

\bibliography{loopaffine.bib}

\nocite*{}

\bibliographystyle{plain}

School of mathematics, Tata Institute of Fundamental Research,
Homi Bhabha Road, Mumbai 400005, India.\\
email: sena98672@gmail.com, senapati@math.tifr.res.in

Department of mathematics and statistics, IIT Kanpur,
Kalyanpur, Kanpur, 208016, India.\\
sachinsh@iitk.ac.in

Department of mathematics and statistics, IIT Kanpur,
Kalyanpur, Kanpur, 208016, India.\\
sudiptaa@iitk.ac.in

\end{document}